\documentclass[12pt]{article}

\usepackage{graphicx}

\usepackage{amsmath}

\usepackage{amssymb}

\usepackage{bm}

\usepackage{amscd}

\usepackage[top=30truemm, bottom=30truemm, left=32truemm, right=32truemm]{geometry}

\usepackage{setspace}
\usepackage{amsthm}

\usepackage{array}

\newtheorem{dfn}{Definition}[subsection]
\newtheorem{pp}[dfn]{Proposition}
\newtheorem{col}[dfn]{Corollary}
\newtheorem{lem}[dfn]{Lemma}

\newtheorem{thrm}[dfn]{Theorem}

\def\sp{\operatorname{Span}}
\def\stb{\operatorname{Stab}}

\newcommand{\1}{\left}
\newcommand{\2}{\right}
\newcommand{\calo}{\mathcal{O}}
\newcommand{\calp}{\mathcal{P}}

\newcommand{\calb}{\mathcal{B}}

\newcommand{\bbc}{\mathbb{C}}
\newcommand{\bbz}{\mathbb{Z}}

\newcommand{\bbg}{\mathbb{G}}

\newcommand{\frako}{\mathfrak{o}}
\newcommand{\frakp}{\mathfrak{p}}
\newcommand{\fraks}{\mathfrak{S}}
\newcommand{\fraka}{\mathfrak{A}}

\newcommand{\al}{\alpha}
\newcommand{\ep}{\varepsilon}

\newcommand{\wh}{\widehat}
\newcommand{\rh}{\wh{R}}
\newcommand{\bs}{\backslash}
\newcommand{\ol}{\overline}
\newcommand{\vp}{V_\calp}
\newcommand{\op}{\omega_\calp}
\newcommand{\ox}{\frako^\times}
\newcommand{\rl}{_{r=0}^{\ell-1}}
\newcommand{\kn}{_{k=1}^N}
\newcommand{\iir}{_{i=1}^{I_r}}
\newcommand{\uel}{^{\langle \ell \rangle}}
\newcommand{\uei}{^{(i)}}
\newcommand{\ueo}{^{(1)}}
\newcommand{\pri}{\varphi_r\uei}
\newcommand{\prj}{\varphi_r^{(j)}}
\newcommand{\pro}{\varphi_r\ueo}
\newcommand{\la}{\langle}
\newcommand{\ra}{\rangle}
\newcommand{\ten}{\;\cdot\;}
\newcommand{\grsi}{\gamma_{r,s}^{(i)}}
\newcommand{\grri}{\gamma_{r,r}^{(i)}}
\newcommand{\grro}{\gamma_{r,r}^{(1)}}
\newcommand{\grzi}{\gamma_{r,0}^{(i)}}
\newcommand{\gsso}{\gamma_{s,s}^{(1)}}
\newcommand{\gzzo}{\gamma_{0,0}^{(1)}}
\newcommand{\xln}{X(\ell,N)}
\newcommand{\tf}{\otimes
 \hspace{-0.16in} {}_{{}_{{}_{{}_{k=1}}}}
 \hspace{-0.16in} {}^{{}^{{}^N}} \;f_k}
\newcommand{\bt}{\bigotimes}
\newcommand{\aia}{\al\in\fraka(n)}
\newcommand{\ari}{A_r\uei}
\newcommand{\iwa}{_{i=I_{{}_{r-1}}+1}^{I_r}}
\newcommand{\ao}{\al\ueo}
\newcommand{\ai}{\al\uei}

\newcommand{\beq}{\begin{equation}}
\newcommand{\eeq}{\end{equation}}
\newcommand{\ba}{\begin{array}}
\newcommand{\ea}{\end{array}}
\newcommand{\nt}{\noindent}
\newcommand{\np}{\newpage}
\newcommand{\vs}{\vspace}
\newcommand{\hs}{\hspace}
\newcommand{\dis}{\displaystyle}

\title{\bf Harmonic analysis on a local field\\ towards addition theorems for\\ multivariate Krawtchouk polynomials}
\author{Koei KAWAMURA}
\date{}

\begin{document}

\maketitle

\begin{abstract}
We aim addition theorems for multivariate Krawtchouk polynomials, following Dunkl\cite{Dun} for 1-variate case.
We work on harmonic analysis on a non-Archimedean local field, that is a group theoretic situation where these polynomials play roles of the zonal spherical functions.
Unlike Dunkl's case, we use decompositions of spherical representations as not necessarily irreducible.
We examine translations of zonal spherical functions, and have a kind of addition theorem for multivariate Krawtchouk polynomials.
\end{abstract}

\section{Introduction}\label{intro}

\nt
Harmonic analysis on groups is one of the most suitable framework to investigate special functions.
An addition theorem for certain orthogonal polynomials may correspond to the expansion of translation of the zonal spherical functions as Fourier series over a subgroup, thus important in order to grasp structure of the homogeneous space and group theoretic meanings of the special functions.
Many cases are studied in this direction since the 1970's.
As for discrete type (orthogonal with respect to a summation), we can see addition theorems for Krawtchouk polynomials by Dunkl\cite{Dun}, Hahn polynomials by Dunkl\cite{Dun2} and $q$-Krawtchouk polynomials by Stanton\cite{Sta}.
We remark that all of these examples are of 1-variate cases.

On the other hand, as seen in recent works, multivariate orthogonal polynomials have also been interpreted in harmonic analysis.
For instance, Mizukawa\cite{Miz} expressed zonal spherical functions on the complex reflection groups in terms of multivariate Krawtchouk polynomials.
Scarabotti\cite{Sca} regarded multivariate Hahn polynomials as intertwining functions (a generalization of zonal spherical functions) on the symmetric groups.
So it will be expected that the application of methods for 1-variate cases may induce addition theorems for multivariate cases, although I couldn't find works in such direction.

In fact, irreducible decompositions are complicated and difficult to parametrize in multivariate cases.
Nevertheless, in spite of not necessarily irreducible, zonal spherical functions can be decomposed as impressive ways. 
In this article, we propose such a situation in the harmonic analysis on a non-Archimedean local field, on which the author himself\cite{ore} induced the zonal spherical functions in terms of the multivariate Krawtchouk polynomials.
And we induce an addition theorem for that polynomials with respect to one kind of decompositions.

\vspace{.1in}
\noindent
\underline{\bf Notations.}\;\; We use notations below throughout the paper, most of which may be standard.
Both notations $|A|$ and $\sharp A$ stand for the cardinarity of a set $A$.
For a finite sequence of numbers $z=\1( z_s\2)_{s \in S}$ indexed by $S\subset \bbz$ and an integer $k\in \bbz$, we use notations
\begin{equation}
 \label{|}
|z|=\sum_{s\in S}z_s,\quad z|^k=\sum_{s\leq k}z_s,\quad z|_k=\sum_{s\geq k} z_s
\end{equation}
(the later two may not be common).
In most cases they are used for $z=(z_r)_{r=0}^{\ell-1}$, then $z|^k=z_0+z_1+\cdots +z_k,\; z|_k=z_k+z_{k+1}+\cdots+z_{\ell-1}$.

When a group $G$ acts on a set $X$, and $x, y \in X$ are in a same $G$-orbit, then we write $x\sim y\; (G)$.

\section{Zonal spherical functions on a non-Archimedean local fields}

Here we summarize well-known statements, and results by the author\cite{ore}.

\subsection{Zonal spherical functions on finite abelian groups}

We consider a situation that a compact group $G$ acts on a finite abelian (additive) group $A$, and denote $\bbg=G\rtimes A$ the semi-derect group of this action.
Then $\bbg$ acts on $A$ by
\beq
\bbg \curvearrowright A,\quad (b,g)\cdot a=b+g(a)\qquad(a,b \in A,\;g\in G).
\eeq
The permutation representation with respect to this action is defined on the space $V=\bbc[A]$ of all complex valued functions on $A$.
We note that it is unitary according to the $L^2$-inner product of $V$;
\beq
  \label{nai}
\langle \varphi, \psi  \rangle=\frac1{|A|} \sum_{a\in A}\varphi(a) \overline{\psi(a)}\qquad(\varphi, \psi \in V).
\eeq
The representation $V$ is decomposed into irreducible subrepresentations with multiplicity free,
and each irreducible component involves exact one dimensional $G$-invariant subspace.
It is one aspect of the fact that the pair $(\bbg, G)$ is a Gelfand pair\cite{Ter}.

Let us explain this decomposition in more detail.
The character group $\wh{A}$ of $A$ allows the contragredient action of $G$.
We denote the set of all orbits of it by $G\bs \wh{A}$.
For each $\calp \in G\bs\wh{A}$, we put $V_\calp := \sp \calp \subset V=\bbc[A]$.
Then mutually orthogonal irreducible decomposition of $V$ is given by
\beq
   \label{decom0}
V=\bigoplus_{\calp \in G\bs\wh{A}} \vp.
\eeq
We call $V_\calp$ the spherical representation of the pair $(\bbg, G)$ associated with the orbit $\calp \in G\bs\wh{A}$.
It has one dimensional $G$-invariant subspace ${V_\calp}^G$ as mensioned above (by Frobenius reciprocity).
The generating element $\op$ of ${\vp}^G$ is called the zonal spherical function; we normarize as $\op(0)=1$.
We may summarize the setting so far as follows:
$$
\bbg=A\rtimes G \curvearrowright V=\bbc[A]=\bigoplus_{\calp \in G\bs\wh{A}} \vp \supset V^G= \bigoplus_{\calp \in G\bs\wh{A}}{\vp}^G=\bigoplus_{\calp \in G\bs\wh{A}}\bbc\op.
$$

\subsection{Representations of wreath products on a non-Archimedean local field}

Let $F$ be a non-Archimedean local field, $v:F\to \bbz\cup \{\infty\}$ be the discrete valuation, $\frako=\{ a\in F \mid v(a)\geq 0\}$ be the ring of integers of $F$, and $\frakp=\{a\in F\mid v(a)\geq 1\}$ be the maximal ideal of $\frako$.
We fix a generator of $\frakp$ denoted by $\pi$.
The residue field $\frako / \frakp$ has the finite order $q$, a power of a prime.
We fix a natural number $\ell\geq1$ and denote by $R:= \frako / \frakp^\ell$ the residue ring of the order $q^\ell$.
The valuation $v:R\to \{0,\cdots,\ell-1\} \cup \{\infty\}$ is
naturally induced.
We concider a wreath product group $G=(\ox)^N \rtimes \fraks_N$, where $\ox$ is the multiplicative group of $\frako$ and $\fraks_N$ is the symmetric group of order $N\geq 1$.
Then $G$ acts on the additive group $A=R^N$ by
\beq
  \label{w-act}
(c,\sigma) a = \1(c_k a_{\sigma^{-1}(k)}\2)_{k=1}^N
\qquad
\1(c=\1(c_k\2)_{k=1}^N \in (\ox)^N,\; \sigma \in \fraks_N,\; a=\1(a_k\2)_{k=1}^N \in R^N
\2).
\eeq 
We take a parameter set
\beq
X(\ell,N)=\1\{x=\1(x_r\2)_{r=0}^{\ell-1}\in\1(\bbz_{\geq 0}\2)^\ell \;\Big|\; |x|\leq N \2\}.
\eeq
Then for $x=\1(x_r\2)\rl\in X(\ell, N)$,
\beq
\calo(x)=\1\{
\1(a_k\2)\kn \in R^N\;\Big|\; \sharp\{k\mid v(a_k)=r\}=x_r,\; 0\leq r\leq\ell-1
\2\}
\eeq
is an orbit of the action (\ref{w-act}),
and all orbits are parametrized as $x$ runs over $X(\ell,N)$.
On the other hand, we define the `dual valuation' $\wh{v}: \wh{R} \to \{-\infty\}\cup \{0,\cdots,\ell-1\}$ on the character group $\wh{R}$ by
\beq
  \label{dv}
\wh{v}(\chi)=\max\1\{ v(a)\mid a \in R,\; \chi(a)\neq 1 \2\},\qquad \text{except for}\;\; \wh{v}(1)=-\infty.
\eeq
We may identify the character group $\wh{A}=\wh{R^N}$ with the direct product $(\wh{R})^N$.
Then as $n=\1(n_r\2)\rl$ runs over $X(\ell, N)$,
\beq
\calp(n)=\1\{
\1(\chi_k\2)\kn\in(\wh{R})^N \;\Big|\; \sharp\{k\mid \wh{v}(\chi_k)=r\}=n_r,\; 0\leq r\leq\ell-1
\2\}
\eeq
completes the list of orbits of the contragredient action of $G$ on $\wh{A}$.
Under the notations above, we have the followings\cite{ore}:

\begin{pp}
The value $\omega_n(x):=\omega_{\calp(n)}\1(\calo(x)\2)$ of the zonal spherical function associated with an orbit $\calp(n)$ on an orbit $\calo(x)$ is expressed as
\beq 
  \label{w-kyu}
\omega_n(x)=K_n\uel\1(x;\! \begin{array}{c} \frac{q-1}{q}\end{array}\!; N  \2),
\eeq
where the right hand side is the $\ell$-variate Krawtchouk polynomial defined in the following subsection.
\end{pp}

\subsection{Multivariate Krawtchouk polynomials}

The 1-variate Krawtchouk polynomials $K_n(x; p; N)$ are discrete orthogonal polynomials with respect to the binomial distribution.
For integers $n\leq N$, a variable $x$ and a probabilistic parameter $0<p<1$, they are expressed by use of the Gaussian hypergeometric function ${}_2F_1$ as
\beq
  \label{kraw}
K_n(x; p; N)= {}_2F_1\1( \begin{array}{c} -n,\;-x\\-N\end{array}  ;
\begin{array}{c} \frac1p \end{array} \2)=\sum_{k=0}^n\frac{(-n)_k(-x)_k}{(-N)_k\, k!p^k},
\eeq
where $(a)_k$ is the Pochhammer symbol; $(a)_0=1$ and $(a)_k=\prod_{s=0}^{k-1}(a+s)\;(k\geq 1)$.

Griffiths\cite{Gri} first gave the multivariate extensions of the Krawtchouk polynomials, which are orthogonal for the multinomial distribution.
In this article we deal with special cases which are expressed in terms of products of 1-variate Krawtchouk polynomials, as defined by Xu\cite{Xu} (ours differ from Xu's in just normalizations).
For $x=\1(x_r \2)\rl,\, n=\1(n_r\2)\rl \in X(\ell, N)$ and a parameter set $\bm{p}=\1(p_r\2)\rl$, we define the $\ell$-variate Krawtchouk polynomials by
\beq
   \label{lkraw}
K_n\uel(x; \bm{p}; N)=\frac1{(-N)_{|n|}}\prod\rl(-N_r)_{n_r} K_{n_r}(x_r ; p_r ; N_r),
\eeq
\beq
\text{where}\quad N_r=N-x|^{r-1}-n|_{r+1}
\eeq
(recall $x|^{r-1}=\sum_{s=0}^{r-1}x_s,$ and $n|_{r+1}=\sum_{s=r+1}^{\ell-1}n_s$ in our notation(\ref{|})).
Especially when $\bm{p}=(p,\cdots,p)$, we denote $K_n\uel(x; p; N):=K_n\uel(x; \bm{p}; N)$, which is used in (\ref{w-kyu}). 
See \cite{Xu} or \cite{ore} for the definition by generating functions and typical properties of the polynomials.

\section{Decomposition of the spherical representations}

\subsection{Harmonic analysis on $R=\frako / \frakp^\ell$ with $\ox$-action}

Here we examine harmonic analysis on $R=\frako / \frakp^\ell$ with the $\ox$-action,
which is the $N=1$ case of, and based on which we will develop harmonic analysis on $R^N$ with the action (\ref{w-act}) of $G=(\ox)^N \rtimes \fraks_N$
.  

The $\ox$-actions on $R$ and on the character group $\wh{R}$ respectively yield $\ell+1$ orbits
\beq
R_r=\{ a \in R \mid v(a)=r \}\quad(r=0,1,\cdots,\ell-1,\infty),
\eeq
\beq
\wh{R}_r=\{ \chi \in \wh{R} \mid \wh{v}(\chi)=r \}\quad(r=-\infty, 0,1,\cdots,\ell-1)
\eeq
according to the values of $v$ and $\widehat{v}$. We can make a correspondence between these orbits as follows:
We fix a character $\theta\in \widehat{R}_{\ell-1}$ and define a group isomorphism
\beq
  \label{iso}
R\longrightarrow \wh{R},\quad a\mapsto \theta(a\;\cdot\;),
\eeq
where $\theta(a\;\cdot\;)$ denotes a map $b \mapsto \theta(ab)$ for $b\in R$.
Then as easily seen, it maps $R_r$ to $\wh{R}_{\ell-r-1}\;(0\leq r \leq \ell-1)$ and $R_{\infty}$ to $\wh{R}_{-\infty}$.
The cardinalities $I_r:=|\wh{R}_r|$ of the orbits are given by
\beq
I_{-\infty}=1,\qquad I_r=q^r(q-1)\;\;(0\leq r\leq \ell-1),
\eeq
since $|{R}_r|=|\frakp^r/\frakp^\ell - \frakp^{r+1}/\frakp^\ell|=q^{\ell-r-1}(q-1)$ and the correspondence above.
We sometimes use $I_{-1}:=1$ for convenience sake.

Since $\wh{R}$ is an orthonormal basis of the space $\bbc[R]$ of all functions on $R$, we have an orthogonal decomposition
\beq
\bbc[R]=W_{-\infty} \oplus \bigoplus\rl W_r,\qquad
\text{where}\;\;\;W_r=\sp{\wh{R}_r}
\eeq
(we remark that $W_{-\infty}=\bbc\cdot 1$ is the space of constant functions on $R$).
As mentioned in (\ref{decom0}), it is the irreducible decomposition as a representation of the semi-direct group $R \rtimes \ox$.
We now aim to decompose each $W_r$ into irreducible (1-dimensional) $\ox$-representations.
Let us define a decreasing sequence $\1\{ \ox_r \2\}_{r=-\infty,0,1,2,\cdots}$ of subgroups of $\ox$ by
\beq
\ox_{-\infty}=\ox,\qquad \ox_r=\1\{ 1+a\pi^{r+1} \mid a \in\frako  \2\}\;\;(r\geq 0)
\eeq
(recall that $\pi$ is a generator of $\frakp$).
We remark that $\ox_r$ is the stabilizer subgroup for an arbitrary $\chi\in\wh{R}_r$.

\begin{pp}
For $r=-\infty, 0,1,\cdots$, there exist just $I_r$ numbers of 1-dimensional $\ox$-representations that are trivial on $\ox_r$.
And $W_r$ is the direct sum of all of them with multiplicity free as an $\ox$-representation.
\end{pp}

\begin{proof}
We have $W_r=\mathrm{ind}_{\ox_r}^{\ox}1$ by the definition of induced representation (because for an arbitrary $\chi \in \wh{R}_r$, $\bbc\chi$ is a trivial representation of $\ox_r$, and actions of representatives of $\ox/\ox_r$ on $\chi$ yield elements of $\wh{R}_r$ once for each).
Frobenius reciprocity therefore gives $[W_r : V]_{\ox}=[ V : 1]_{\ox_r}$ for an arbitrary 1-dimensional $\ox$-representation $V$.
It asserts that $V$ occurs in $W_r$ just once if and only if $V$ is trivial on $\ox_r$.
Since $\mathrm{dim}\, W_r=I_r$, we conclude the statement.
\end{proof}

We thus let $\{\xi\uei\}\iir$ be all characters of $\ox$ which are trivial on $\ox_r$ (In particular, $\xi\ueo$ is the trivial character of $\ox$).
Then we have the $\ox$-irreducible decomposition
\beq
W_r=\bigoplus\iir\bbc\pri,
\eeq
where $\pri$ is the generator of $\xi\uei$-component of $W_r$, normarized as $\langle \pri, \pri \rangle=1$
with respect to the inner product (\ref{nai}) of $\bbc[R]$.
It means that $\pri$ is a relatively $\ox$-invariant function on $R$:
\beq
  \label{ri}
\pri(c^{-1}a)=\xi\uei(c)\pri(a)\quad(c\in\ox,\; a\in R).
\eeq 
Furthermore when $1\leq i<j\leq I_r$, functions $\pri$ and $\prj$ are orthogonal.
In fact, using element $c\in \ox$ such that $\xi\uei(c)\overline{\xi^{(j)}(c)}\neq 1$, we have
$$
\begin{array}{rl}\displaystyle
\langle\pri,\prj\rangle
&\displaystyle
=\frac1{|R|}\sum_{a\in R}\pri(a)\overline{\prj(a)}
=\frac1{|R|}\sum_{a\in R}\pri(c^{-1}a)\overline{\prj(c^{-1}a)}\\
&\displaystyle
=\xi\uei(c)\overline{\xi^{(j)}(c)}\langle\pri,\prj\rangle,
\end{array}
$$
which implies $\langle\pri,\prj\rangle=0$.
Consequently $\bbc[R]$ acquires an orthonormal basis
\beq
  \label{ob1}
\{1\}\cup\{\pri\mid0\leq r \leq \ell-1,\;1\leq i\leq I_r\}.
\eeq

By the definition, $\pro$ is the unique $\ox$-invariant normarized function belonging to $W_r$.
We thus obviously have
\beq
  \label{pro}
\pro=\frac1{\sqrt{I_r}}\sum_{\chi\in\wh{R}_r}\chi.
\eeq
As for $i\geq2$, coefficients of $\pri$ in the expansion over $\wh{R}_r$ satisfy the following properties (we will use it as a lemma for Proposition \ref{maru}).

\begin{lem}
  \label{kage}
Let $0 \leq r \leq \ell-1,\; 2\leq i \leq I_r$, and put
\beq
  \label{ec}
\varphi_r^{(i)}=\sum_{\chi\in \widehat{R}_r}\varepsilon_\chi \chi
\eeq
with coefficients $\varepsilon_\chi \in\bbc$. We also fix $u$ as an integer $0\leq u \leq r$ or $u=-\infty$. Then

\vs{0.05in}
$(1)$\quad $\dis \sum_{\chi\in \wh{R}_r} |\ep_\chi|^2=1$ hold.

$(2)$\quad If $i\leq I_u$ and $\chi, \chi' \in\wh{R}_r$ are in a same $\ox_u$-orbit, then we have $\ep_\chi=\ep_{\chi'}$.

\vs{0.05in}
$(3)$\quad If $i>I_u$, then for any $\ox_u$-orbit $\calp\subset \wh{R}_r$, we have $\dis \sum_{\chi\in\calp} \ep_\chi=0$.
\end{lem}

\begin{proof}
(1) is obvious from $\la \pri, \pri \ra=1$.
Now by the reason of (\ref{ri}), we remark that the function $\pri$ is $\ox_u$-invariant when and only when $1\leq i \leq I_u$.
For (2), we let an element $c\in \ox_u$ such that $\chi'=\chi(c^{-1}\ten)$ act on the both sides of (\ref{ec}), and compare the coefficients on $\chi'$.
Then we shall get $\ep_{\chi'}=\ep_{\chi}$.

For (3), we consider a decomposition
\beq
W_r=W_\calp\oplus W_\calp^\perp,\quad
W_\calp=\bigoplus_{\chi\in\calp}\bbc\chi
\;\;\text{and}\;\;
W_\calp^\perp=\bigoplus_{\chi \in \wh{R}_r - \calp} \bbc\chi,
\eeq
that is stable under the $\ox_u$-action.
We note the $W_\calp$-component of $\pri$ is $\dis\varphi_\calp:=\sum_{\chi \in \calp}\ep_\chi\chi$.
And we can take an element $c\in \ox_u$ such that $\xi\uei(c)\neq1$, as $i>I_u$.
When $c$ acts, since $\pri$ is multiplied by $\xi\uei(c)$,
so is $\varphi_\calp$; so we have
\beq
\varphi_\calp(c^{-1}a)=\xi\uei(c)\varphi_\calp(a) \quad(a\in R).
\eeq
By substituting $a=0$ in the both sides, we have $\varphi_\calp(0)=0$, that is $\dis \sum_{\chi\in\calp}\ep_\chi=0$.
\end{proof}

At the last of this part, we study $R$-actions on functions on $R$, namely the translations.
Let us consider an element $\pi^s := \pi^s + \frakp^\ell \in R$ for $0\leq s \leq \ell-1$.
We especially examine an action of $\pi^s$ on the function $\varphi_r^{(1)}$.
Since $W_r$ is an $R\rtimes \frako^\times$-representation and has a decomposition (20),
we know so far that the translated function $\varphi_r^{(1)}(\pi^s + \,\cdot\;)$ can be written as a linear combination of $\{ \varphi_r^{(i)} \}_{i=1}^{I_r}$.

\begin{pp}
  \label{maru}
Let $0\leq s,r \leq \ell-1$.

$(1)$\quad If $s<r$, we can write $\dis \pro(\pi^s+\ten)=\sum_{i=I_{r-s-1}+1}^{I_{r-s}}\grsi\pri$,
with coefficients $\grsi$ satisfying
\beq
 \label{gam}
|\grsi|=\frac1{\sqrt{q^{r-s-1}}(q-1)}.
\eeq

$(2)$\quad When $s=r$, we can write $\dis \pro(\pi^r+\ten)=\sum_{i=1}^{q-1}\grri\pri$,
with coefficients $\grro=-\frac1{q-1}$ and $\grri\;(2\leq i\leq q-1)$ satisfying $(\ref{gam})$.

\vs{0.05in}
$(3)$\quad If $s>r$, we have $\pro(b+\ten)=\pro$ for an arbitrary $b\in R_s$.
\end{pp}

\begin{proof}
First we prove (3),
so let $s>r$ and $b\in R_s$.
By the definition of dual valuation $\wh{v}$ in (\ref{dv}),
we have $\chi(b+\ten)=\chi$ for all $\chi \in \wh{R}_r$.
So the statement is obvious by the form of $\pro$ in (\ref{pro}).

From now on we assume $s\leq r$. As remarked in the proof of the last lemma, the $\frako^\times_{r-s}$-invariant subspace of $W_r$ is $\dis W_r^{\frako^\times_{r-s}}=\bigoplus_{i=1}^{r-s}\bbc\pri$.
Also, $\pro(\pi^s+\ten)$ is $\ox_{r-s}$-invariant.
In fact, for $1+a\pi^{r-s+1} \in \ox_{r-s}\;(a \in \frako)$ and $b\in R$, we have
$$
\begin{array}{rl}
\pro\big(\pi^s+(1+a\pi^{r-s+1})b\big)
&=\pro\big( (1+a\pi^{r-s+1})(\pi^s+b)-a\pi^{r+1} \big)\\
&=\pro\big( (1+a\pi^{r-s+1})(\pi^s+b) \big)
=\pro(\pi^s+b)
\end{array}
$$
(the second equality is by (3) of this proposition).
It permits us awhile to write
\beq
  \label{ittan}
\pro(\pi^s+\ten)=\sum_{i=1}^{I_{r-s}}\grsi\pri\quad(\grsi\in\bbc),
\eeq
where the coefficients $\gamma_{r,s}^{(i)}\in\bbc$ have an expression
\beq
  \label{grsi}
\grsi=\frac1{\sqrt{I_r}} \overline{\xi\uei(-1)} \overline{\pri(\pi^s)}.
\eeq
In fact, writing $\pri=\sum_{\chi\in\wh{R}_r} \ep_\chi\chi$ as in the last lemma, we have
$$
\begin{array}{rl}
\overline{\grsi}
&\dis =\1\la \pri, \pro(\pi^s+\ten) \2\ra
=\1\la \sum_{\chi\in\wh{R}_r} \ep_\chi\chi,\;
\frac1{\sqrt{I_r}}\sum_{\chi\in\wh{R}_r}\chi(\pi^s)\chi\2\ra
\\[0.3in]
&\dis =\frac1{\sqrt{I_r}}\sum_{\chi\in\wh{R}_r}\ep_\chi\chi(-\pi^s)
=\frac1{\sqrt{I_r}}\pri(-\pi^s)
=\frac1{\sqrt{I_r}}\xi\uei(-1)\pri(\pi^s).
\end{array}
$$

To prove remainings of the statements, we prepare notations and a lemma. For $0\leq s \leq \ell-1$, we denote
\beq
R|_s:=\bigcup_{u\geq s} R_u=\frakp^s/\frakp^\ell,
\qquad
\wh{R}|^s:=\bigcup_{u\leq s}\wh{R}_u,
\eeq
where each union also includes $R_{\infty}=\{0\}$ or
$\wh{R}_{-\infty}=\{1\}$.
These are subgroups respectively of $R$ and $\wh{R}$.

\begin{lem}
  \label{sisi}
Let $0\leq s\leq r\leq \ell-1$.
For $\chi, \chi'\in \wh{R}_r$, the following conditions are equivalent:

$($i$)$\quad $\chi\sim\chi'\;(\ox_{r-s})$\quad(that is, they are in a same $\ox_{r-s}$-orbit).

\vs{.05in}
$($ii$)$\quad $\chi=\chi'$ on $R|_s$.

\vs{.05in}
$($iii$)$\quad $\chi\chi'^{-1}\in \wh{R}|^{s-1}$.
\end{lem}

\begin{proof}
({\it i}) $\Rightarrow$ ({\it ii}) and ({\it ii}) $\Rightarrow$ ({\it iii}) may be easy.
({\it iii}) $\Rightarrow$ ({\it i}) may be justified by using the isomorphism (\ref{iso}) that preserves $\ox_{r-s}$-orbits.
We leave details to the reader.
\end{proof}

We also remark that there is a one-to-one correspondence between $\wh{R}|^s$ and the character group $(\frako / \frakp^{s+1})\,\hat{}$ intertwined by the canonical surjection $\rho:R=\frako / \frakp^\ell \to \frako / \frakp^{s+1}$;
\beq
   \label{rho-taiou}
(\frako / \frakp^{s+1})\,\hat{} \to \wh{R}|^s,\quad \chi \mapsto \chi\circ \rho.
\eeq

Now we prove that if $s<r$ and $i\leq I_{r-s-1}$, then $\grsi$ in (\ref{ittan}) is 0.
Because of (\ref{grsi}), it is the same as $\pri(\pi^s)=0$.
By Lemma \ref{kage}\;(2), we can write
\beq
  \label{}
\pri=\sum_{\calp \in \ox_{r-s-1}\bs\wh{R}_r} \ep_\calp \sum_{\chi \in\calp}\chi\quad(\ep_\calp \in \bbc).
\eeq
So it suffices to prove that $\sum_{\chi\in\calp} \chi(\pi^s)=0$ for an arbitrary $\ox_{r-s-1}$-orbit $\calp \subset \wh{R}_r$.
We take a fixed $\chi_0\in\calp$, then by the equivalence of ({\it i}) and ({\it iii}) in Lemma \ref{sisi},
we can write
\beq
\calp=\1\{ \chi_0\chi \;\Big|\; \chi \in \wh{R}|^s \2\}.
\eeq
We therefore have, using the correspondence (\ref{rho-taiou}),
$$
\sum_{\chi\in\calp} \chi(\pi^s)
= \sum_{\chi \in \rh|^s} \chi_0(\pi^s)\chi(\pi^s)
= \chi_0(\pi^s)\sum_{\chi\in(\frako / \frakp^{s+1})\,\hat{}} \chi(\pi^s)=0.
$$

We can also prove $\gamma_{r,r}^{(1)}= - \frac1{q-1}$ by the use of correspondence (31), as
$$
\grro=\1\la \pro(\pi^r+\ten), \pro \2\ra
=\frac1{q^r(q-1)} \sum_{\chi\in\rh_r} \chi(\pi^r), \quad\text{and}
$$
$$
\sum_{\chi\in\rh_r}\chi(\pi^r)
= \sum_{\chi\in\rh|^r}\chi(\pi^r)-\sum_{\chi\in\rh|^{r-1}}\chi(\pi^r)
=\sum_{\chi\in(\frako / \frakp^{r+1})\,\hat{}} \chi(\pi^r)-\sum_{\chi\in(\frako / \frakp^{r})\,\hat{}} 1
=-q^r.
$$

Now it only remains to prove (\ref{gam}) no matter $s<r$ or $s=r$.
By (\ref{grsi}), we know $\dis |\grsi|=\frac1{\sqrt{I_r}}|\pri(\pi^s)|$.
And since $\pri$ is relatively $\ox$-invariant, we have $\dis |\pri(\pi^s)|=|\pri(b)|$ for any $b\in R_s$.
We write $\dis \pri=\sum_{\chi \in\rh_r}\ep_\chi\chi$ as in Lemma \ref{kage}, then
\beq
  \label{chubho}
|R_s||\pri(\pi^s)|^2
=\sum_{b\in R_s}|\pri(b)|^2
=\sum_{\chi, \chi' \in \rh_r}\ep_\chi \overline{\ep_{\chi'}}\sum_{b\in R_s} \chi(b) \ol{\chi'(b)}.
\eeq
Here we note for $\chi, \chi' \in \rh_r$,
\beq
\sum_{b\in R_s} \chi(b)\ol{\chi'(b)}=
\begin{cases}
q^{\ell-s-1}(q-1) & \text{if}\;\; \chi \sim \chi'\;(\ox_{r-s}),\\
-q^{\ell-s-1} & \text{if}\;\; \chi \nsim \chi'\;(\ox_{r-s}) \;\;\text{and}\;\; \chi\sim\chi'\;(\ox_{r-s-1}),\\
0 & \text{if}\;\; \chi \nsim \chi'\;(\ox_{r-s-1})
\end{cases}
\eeq
(when $s=r$, the reader should replace $\ox_{r-s-1}$ with $\ox=\ox_{-\infty}$ here and afterwards).
In fact, the first case is obvious since $\chi=\chi'$ on $R_s$ (Lemma \ref{sisi}) and $|R_s|=q^{\ell-s-1}(q-1)$.
The second and third cases are respectively by Lemma \ref{sisi},
$$
\sum_{b\in R_s}\chi(b)\ol{\chi'(b)}
=\sum_{b\in R|_s}\chi(b)\ol{\chi'(b)}-\sum_{b\in R|_{s+1}}\chi(b)\ol{\chi'(b)}
=\begin{cases}
-\big|R|_{s+1}\big|=-q^{\ell-s-1},\\ 0.
\end{cases}
$$
We also remark that Lemma \ref{kage} provides $\dis (1)\; \sum_{\chi\in\rh_r}|\ep_\chi|^2=1, \;
(2)\; \ep_{\chi'}=\ep_\chi$ if $\chi'\sim\chi\;(\ox_{r-s})$
and $\dis(3)\; \sum_{\chi'\sim\chi\;(\ox_{r-s-1})}\ep_{\chi'}=0$, since $I_{r-s-1}<i\leq I_{r-s}$ now.
The right hand side of (\ref{chubho}) is therefore,
$$
\begin{array}{l}\dis
q^{\ell-s-1}(q-1)\sum_{\chi\in\rh_r}\ep_\chi \sum_{\chi'\sim\chi\,(\ox_{r-s})}\ol{\ep_{\chi'}}
-q^{\ell-s-1}\1( \sum_{\chi'\sim\chi\,(\ox_{r-s-1})} \ol{\ep_{\chi'}}-
\sum_{\chi'\sim\chi\,(\ox_{r-s})} \ol{\ep_{\chi'}}
\2)\\[0.4in]\dis
= q^{\ell-s-1}(q-1)\cdot q^s\sum_{\chi\in\rh_r}|\ep_\chi|^2+
q^{\ell-s-1}\cdot q^s\sum_{\chi\in\rh_r}|\ep_\chi|^2
=q^\ell.
\end{array}
$$
It concludes
$$
|\grsi|=\frac1{\sqrt{I_r}}|\pri(\pi^s)|
=\frac1{\sqrt{I_r}}\sqrt{\frac{q^\ell}{|R_s|}}
=\frac1{\sqrt{q^{r-s-1}}(q-1)}.
$$

\vs{-0.2in}
\end{proof}

\subsection{Decomposition of $V_n$ along relatively $\ox$-invariant functions on R}

As described in Section 1, any $G=(\ox)^N \rtimes \fraks_N$-orbit $\calp(n)$ in $\wh{R^N}$ is parametrized by $n\in\xln$, and defines the spherical representation $V_n:=V_{\calp(n)}=\sp\calp(n) \subset \bbc[R^N]$ of $\bbg=R^N \rtimes G$.
We remark that
\beq
\dim V_n=|\calp(n)|=\binom{N}n\prod\rl I_r{}^{n_r}
=\binom{N}n q^{\sum\rl rn_r}(q-1)^{|n|},
\eeq
where $\dis \binom{N}n=\frac{N!}{(N-|n|)! \prod\rl n_r!}$ is the multinomial coefficient.
Here we construct $G$-subrepresentations of $V_n$ by the use of relatively $\ox$-invariant functions $\pri\;(0\leq r\leq \ell-1, 1\leq i \leq I_r)$ defined in the last subsection ( we do not reach the irreducible decomposition).

For $n=(n_r)\rl\in\xln$, we define the following parameter set:
\beq
  \label{an}
\fraka(n)=\1\{
\al=(\al_r)\rl \;\;\big|\; \al_r=(\al_r\uei)\iir \in (\bbz_{\geq 0})^{I_r},\;|\al_r|=n_r\;(\text{for each}\;\; r)
\2\}
\eeq
(that is, $\alpha \in \mathfrak{A} (n)$ is a $\ell$-tuple, each member $\alpha_r$ of which is a so-called composition of $n_r$).
And for $\alpha\in\mathfrak{A}$, we define a subspace $V_{n,\alpha}$ of $V_n$ as the span of
\beq
  \label{ba}
\calb(\al)=\1\{
\bigotimes\kn f_k\;\;\Bigg|\;\;
\begin{array}{l}
\text{ $\{f_k\}\kn$ consists of $N-|n|$ numbers of 1 and} 
\\
\text{ $\al_r\uei$ numbers of $\pri\;(0\leq r \leq \ell-1,\; 1 \leq i \leq I_r)$.} 
\end{array}
\2\},
\eeq
where the tensor product 
$\otimes
\hspace{-0.16in} {}_{{}_{{}_{{}_{k=1}}}}
\hspace{-0.16in} {}^{{}^{{}^N}} \;
f_k$
is realized as a function on $R^N$ by
\beq
\1( \bigotimes\kn f_k \2)(a)
=\prod\kn f_k(a_k)
\quad
\1( a=(a_k)\kn\in R^N \2).
\eeq
Let us confirm that each $\otimes
\hspace{-0.16in} {}_{{}_{{}_{{}_{k=1}}}}
\hspace{-0.16in} {}^{{}^{{}^N}} \;
f_k$ satisfying the condition of (\ref{ba}) is in $V_n$.
In fact, we expand it in terms of $\1\{\;\otimes
\hspace{-0.16in} {}_{{}_{{}_{{}_{k=1}}}}
\hspace{-0.16in} {}^{{}^{{}^N}} \;
\chi_k \;\big|\; \chi_k \in \widehat{R}\,
\2\}$, a basis of $\bbc[R^N]$.
Then for $0\leq r \leq \ell-1$, the number of $\chi_k \in \widehat{R}_r$ are $|\alpha_r|=n_r$ in any terms, because $\varphi_r^{(i)} \in W_r =$ Span\,$\widehat{R}_r$.
We also remark that \;\;$\cup
\hspace{-0.2in} {}_{{}_{{}_{{}_{\alpha \in \mathfrak{A}(n)}}}}
\mathcal{B}(\alpha)
$ is an orthogonal system of $V_n$ with respect to the inner product (\ref{nai}) of $\bbc[R^N]$, following the orthogonality of (\ref{ob1}).
By counting the cardinality of bases $\mathcal{B}(\alpha)$, we have
\beq
\dim V_{n,\al}=\binom{N}{n}\prod\rl\binom{n_r}{\al_r},
\eeq
where $\dis \binom{n_r}{\al_r}=\frac{n_r!}{\prod\iir\al_r\uei}$ is the multinomial coefficient.
We can make sure an orthogonal direct sum
\beq
  \label{vnd}
V_n=\bigoplus_{\al \in \fraka(n)} V_{n,\al}
\eeq
holds by comparing dimensions of the both sides;
$$
\ba{r}
\dis
\sum_{\al\in\fraka(n)}\dim V_{n,\al}=
\binom{N}n \sum_{\al\in\fraka(n)}\prod\rl\binom{n_r}{\al_r}
=\binom{N}n\prod\rl\sum_{|\al_r|=n_r}\binom{n_r}{\al_r}
\\ \dis
=\binom{N}n\prod\rl I_r{}^{n_r}=\dim V_n
\ea
$$
(the third equality is due to the multinomial theorem).
Actually, (\ref{vnd}) turns out to be a decomposition into $G$-subrepresentations:

\begin{pp}
For any $\al \in\fraka(n)$, $V_{n,\al}$ is a $G$-subrepresentation of $V_n$.
\end{pp}

\begin{proof}
We just need to check that $V_{n,\al}$ is stable under the $(\ox)^N$- and $\fraks_N$-actions separately.
Firstly, any generator $\tf \in \calb(\al)$ is relatively $(\ox)^N$-invariant, since its components $f_k$ are all relatively $\ox$-invariant.
Secondly, the $\fraks_N$-actions just cause permutations of the components $f_k$'s of $\tf$, so it is still in $V_{n,\al}$.
Observations above prove the statement.
\end{proof}

We prepare a notation for the symmetrizations of functions on $R^N$:

\begin{dfn}
Let $m\leq N$ be an integer and $\{ f_k \}_{k=1}^m$ be a set of functions on $R$ (allowing some of which might be the same).
Then we put $f_k=1$ for $k=m+1,\cdots, N$ and define
\beq
  \label{sym}
\1[\bigotimes_{k=1}^m f_k \2]_{\fraks_N}
=\sum_{\sigma \in \fraks_N / \,\stb(\otimes f_k)}
\bigotimes\kn f_{\sigma^{-1}(k)},
\eeq
where $\stb(\otimes f_k) \subset \fraks_N$ is the stabilizer subgroup of $\tf$.
\end{dfn}

As introduced in Section 1.1, the zonal spherical function $\omega_n$ corresponding to a parameter $n\in \xln$ is the unique $G$-invariant normarized function belonging to the spherical representation $V_n$.
Using the notation above, we can explicitly write down $\omega_n$ as follows;
\beq
  \label{wn}
\omega_n=\frac1{\binom{N}n Q}
\1[  
\bigotimes\rl \pro{}^{\otimes n_r}
\2]_{\fraks_N},
\eeq
where $\dis Q=\prod\rl \sqrt{I_r{}^{n_r}}$
\;(
since $\pro(0)=\sqrt{I_r}$ from (\ref{pro}), we need to normarize by it).

We take $u=(u_s)_{s=0}^{\ell-1} \in \xln$ and put
\beq
  \label{piu}
\pi^u:=(
\underbrace{1,\cdots,1}_{u_0\,\text{tuple}},
\underbrace{\pi,\cdots,\pi}_{u_1\,\text{tuple}},
\cdots,
\underbrace{\pi^{\ell-1}, \cdots, \pi^{\ell-1}}_{u_{\ell-1}\,\text{tuple}},
0,\cdots, 0
) \in R^N.
\eeq
We examine components of the translated function $\omega_n(\pi^u+\ten) \in V_n$
with respect to the decomposition (\ref{vnd}).
Here and afterwards we use the following notation:

\vs{0.1in}
\nt\underline{Notation.}\; For any function $f\in\bbc[R^N]$,
let $V_{n,\al}(f)$ denote the $V_{n,\al}$-component of $f$.

\begin{pp}
  \label{wau}
For $\al\in\fraka(n)$ and $u\in\xln$, we define a parameter set
$$
\mathcal{W}(\al,u)=\1\{\,
w=\1(w_{r,s}\2)_{0\leq r\leq s\leq \ell-1} \in (\bbz_{\geq0})^{\binom{\ell+1}{2}}
\;\Big|\; \text{satisfying the conditions below\;}
\2\}
$$

$($i$)$\quad $\dis \sum_{s=r}^{\ell-1}w_{r,s}\leq \al_r\ueo$
\;\; for $0\leq r \leq \ell-1$, 

$($ii$)$\quad $\dis \sum_{r=0}^{s}w_{r,s}\leq 
u_s-\sum_{r=s}^{\ell-1}\,\sum_{i=I_{{}_{r-s-1}}+1}^{I_{r-s}}
\al_r\uei$\;\;
for $0\leq s \leq \ell-1$, and 

$($iii$)$\quad $\dis |u|+|\al\ueo|-N\leq|w|$,\;\;
where $\al\ueo:=(\al\ueo_r)\rl$.

\vs{.2in}
\nt
Then we have
\beq
  \label{vnawnu}
\ba{l}\dis
V_{n,\al}\big(
\omega_n(\pi^u+\ten)
\big)
=\frac1{\binom{N}n Q}\1(
\prod_{0\leq s\leq r\leq \ell-1}\;
\prod_{i=I_{{}_{r-s-1}}+1}^{I_{r-s}} \grsi{}^{\al_r\uei}
\2)
\sum_{w\in \mathcal{W}(\al,u)}
\1(
\prod\rl\grro{}^{w_{r,r}}
\2)
\\ \dis
\times \bigotimes_{s=0}^{\ell-1}\1[
\bigotimes_{r=0}^s \pro{}^{\otimes w_{r,s}}
\otimes \bt_{r=s}^{\ell-1} \;\bt_{i=I_{{}_{r-s-1}}+1}^{I_{r-s}} \pri{}^{\otimes \al_r\uei} 
\2]_{\fraks_{u_s}}
\otimes
\1[
\bt\rl \pro{}^{\otimes\1(\al_r\ueo-\sum_{s=r}^{\ell-1}w_{r,s}  \2)}
\2]_{\fraks_{N-|u|}}
\ea
\eeq
(when $r=s$, we put $I_{r-s-1}=I_{-1}:=1$).
\end{pp}

\begin{proof}
We put
\beq
 (r_k)\kn=(
-\infty,\cdots,-\infty,
\underbrace{0,\cdots,0}_{n_0\,\text{tuple}},
\underbrace{1,\cdots,1}_{n_1\,\text{tuple}},
\cdots,
\underbrace{\ell-1, \cdots, \ell-1}_{n_{\ell-1}\,\text{tuple}}
),
\eeq
and regard the product group $\fraks_n:=\fraks_{N-|n|}
\times \fraks_{n_0} \times \cdots \times \fraks_{n_{\ell-1}}$ as its stabilizer subgroup in $\fraks_N$.
Then from (\ref{wn}) we can write
\beq
  \label{wnu}
\omega_n(\pi^u+\ten)=\frac1{\binom{N}nQ}
\sum_{\sigma\in \fraks_N / \fraks_n}\1\{
\bt_{s=0}^{\ell-1} \bt_{k=u|^{s-1}+1}^{u|^s}
\varphi_{r_{\sigma^{-1}(k)}}\ueo(\pi^s+\ten)\otimes \bt_{k=|u|+1}^N \varphi_{r_{\sigma^{-1}(k)}}\ueo
\2\}
\eeq
(where we regard $\varphi_{-\infty}^{(1)}=1$).
Each tensor component $\varphi_r^{(1)}(\pi^s + \,\cdot\;)$ in here is, by Proposition 2.1.3, either expanded as $\Sigma
\hspace{-0.07in} {}_{{}_{{}_{{}_{i}}}}
\; \gamma_{r,s}^{(i)} \varphi_r^{(i)}
$ (if $s \leq r$), or equal to $\varphi_r^{(1)}$ (if $s>r$, including the case of $r=-\infty$).
We expand (\ref{wnu}) by taking $\gamma_{r,s}^{(i)} \varphi_r^{(i)}$ one by one component.
Then each term belongs to $V_{n,\beta}$ for some $\beta \in \mathfrak{A}(n)$.
For a fixed $\alpha\in\mathfrak{A}(n)$, thus $V_{n,\alpha}\1( \omega_n(\pi^s+\,\cdot\;) \2)$ is just the sum of terms belonging to $V_{n,\alpha}$ in such the expansion.

For $\sigma \in \fraks_N / \fraks_n$, we denote by
\beq
  \label{ts1}
 T(\sigma)=\bt_{s=0}^{\ell-1}\,\bt_{k=u|^{s-1}+1}^{u|^s}
\varphi_{r_{\sigma^{-1}(k)}}\ueo(\pi^s+\ten)\otimes \bt_{k=|u|+1}^N \varphi_{r_{\sigma^{-1}(k)}}\ueo,
\eeq
the term corresponding to $\sigma$ in (\ref{wnu}),
and consider a condition that
\beq
  \label{ts2}
\text{$T(\sigma)$ has the non-zero $V_{n,\al}$-component in its expansion.}
\eeq
We remark that this condition is stable under the left $\fraks_u$-action (where $\fraks_u:=\fraks_{u_0}\times\cdots\times
\fraks_{u_{\ell-1}}\times\fraks_{N-|u|}$ is the stabilizer subgroup of $\pi^u$ in $\fraks_N$).
In fact, for any $\tau\in\fraks_u$, the change between $T(\sigma)$ and $T(\tau\sigma)$ is just a permutation within $(u|^{s-1}+1)$-th to $u|^s$-th for some $0\leq s\leq \ell-1$, or $(|u|+1)$-th to $N$-th tensor components in (\ref{ts1}).

So we can define a set
\beq
  \label{sau}
S(\al,u)=\{ \sigma\in \fraks_u\bs \fraks_N / \fraks_n\mid \text{$T(\sigma)$ satisfies (\ref{ts2})} \}.
\eeq
We now make a one-to-one correspondence between $S(\al,u)$ and $\mathcal{W}(\al,u)$.
Let $\sigma \in S(\al,u)$.
Then for any $0\leq s\leq \ell-1$, the $u_s$-tuple $(r_{\sigma^{-1}(k)})_{k=u|^{s-1}+1}^{u|^s}$
includes exactly $\dis \sum_{i=I_{{}_{r-s-1}}+1}^{I_{r-s}} \al_r\uei$ numbers of each $r>s$, and more than $\dis \sum_{i=2}^{q-1} \al_s\uei$ numbers of $s$.
Because by Proposition \ref{maru}, $\pri$ for $r\geq s,\; I_{r-s-1}+1\leq i\leq I_{r-s}$ appears only in the expansion of $\pro(\pi^s+\ten)$,
and conversely the expansion has just such terms except for the $r=s$ case which has $\varphi_s\ueo$.
We hence put
\begin{eqnarray}
&&w_{s,s}=\sharp\{ k\mid r_{\sigma^{-1}(k)}=s,\; u|^{s-1}+1\leq k\leq u|^s \}
- \sum_{i=2}^{q-1}\al_s\uei,\quad\text{and}
\\[.0in]
&&w_{r,s}=\sharp\{ k\mid r_{\sigma^{-1}(k)}=r,\; u|^{s-1}+1\leq k\leq u|^s \}\quad\text{for $r<s$.}
\end{eqnarray}
We remark, for $r\leq s$, that $w_{r,s}$ represents the number of $\pro$ among $(u|^{s-1}+1)$-th to $u|^s$-th tensor components of each term of $T(\sigma)$ belonging to $V_{n,\al}$.
They therefore satisfy the condition ({\it i}).
Also ({\it ii}) is obvious from how we took them.
For each $0\leq r\leq \ell-1$, the number of $r$ among $(r_{\sigma^{-1}(k)})_{k=1}^{|u|}$ is now
\beq
\sum_{s=0}^r \; \sum_{i=I_{{}_{r-s-1}}+1}^{I_{r-s}} \al_r\uei + \sum_{s=r}^{\ell-1}w_{r,s}=n_r-\al_r\ueo+\sum_{s=r}^{\ell-1}w_{r,s},
\eeq
hence that among $(r_{\sigma^{-1}(k)})_{k=|u|+1}^N$ is $\dis\al_r\ueo-\sum_{s=r}^{\ell-1}w_{r,s}$.
So we have
\beq
\sum\rl\1(
\al_r\ueo-\sum_{s=r}^{\ell-1}w_{r,s}
\2)
\leq N-|u|,
\eeq
which induces the condition ({\it iii}).

Conversely when $w=(w_{r,s})\in\mathcal{W}(\al,u)$ is given, for $0\leq s\leq \ell-1$, we let a $u_s$-tuple $(r_{\sigma^{-1}(k)})_{k=u|^{s-1}+1}^{u|^s}$ consist of $\dis \sum_{i=I_{{}_{r-s-1}}+1}^{I_{r-s}}\al_r\uei$ numbers
of each $r>s$,\;
$\dis \sum_{i=2}^{q-1}\al_s\uei +w_{s,s}$ numbers of $s$,
\;$w_{r,s}$ numbers of each $r<s$, and $-\infty$'s for the remainings.
Also let a $(N-|u|)$-tuple $(r_{\sigma^{-1}(k)})_{k=|u|+1}^{N}$ consist of $\dis \al_r\ueo-\sum_{s=r}^{\ell-1}w_{r,s}$ numbers of each $r$ and $-\infty$'s for the remainings.
Such $\sigma$ is uniquely determined as $\sigma \in S(\al,u)$.
The process above gives a one-to-one correspondence between $S(\al,u)$ and $\mathcal{W}(\al,u)$.

We have already explained how to take terms belonging to $V_{n,\al}$ in the expansion of (\ref{wnu}).
That is, considering interpretations of $w\in \mathcal{W}(\al,u)$ in the explanation above, we have
$$
\ba{rl}\dis
V_{n,\al}\big(
\omega_n(\pi^u+\ten)
\big)%
=
&\dis
\frac1{\binom{N}n Q}
\sum_{\sigma\in S(\al,u)}\1[
V_{n,\al}(T(\sigma))
\2]_{\fraks_u}
\\ \dis
=
&\dis
\frac1{\binom{N}n Q}
\sum_{w\in\mathcal{W}(\al,u)}
\1[
\bt_{s=0}^{\ell-1}\1(
\bt_{r=0}^{s-1}\pro{}^{\otimes w_{r,s}} \otimes
(\gamma_{s,s}\ueo \varphi_s\ueo)^{\otimes w_{s,s}}
\2.\2.
\\ &\dis
\1.
\otimes \bt_{r=s}^{\ell-1}\,\bt_{i=I_{{}_{r-s-1}}+1}^{I_{r-s}}
(\grsi\pri)^{\otimes \al_r\uei}\otimes 1 \otimes\cdots\otimes1
\2)
\\ &\dis
\1.
\otimes\1(
\bt\rl\pro{}^{\otimes \1( \al_r\ueo - \sum_{s=r}^{\ell-1}w_{r,s} \2)}\otimes1 \otimes\cdots\otimes1
\2)
\2]_{\fraks_u},
\ea
$$
which is equal to the statement.
\end{proof}

\section{An addition theorem for multivariate Krawtchouk polynomials}

\subsection{The basic idea for addition theorems}

Here following Dunkl\cite{Dun}, we exhibit the basic idea to induce addition theorems of zonal spherical functions by the use of harmonic analysis in the setting of Section 1.1 (that a compact group $G$ acts on an finite abelian group $A$).
The procedure is as below:

\vs{.1in}
(i)\quad We decompose a spherical representation $V_\calp$ (that is irreducible as a representation of $\bbg=A\rtimes G)$ as a representation of the subgroup $G\subset \bbg$.

\vs{.03in}
(ii)\quad We let a particular element $a\in A$ act on the zonal spherical function $\op$, that is, consider a function $\op(a+\ten)\in \vp$, and decompose it into components according to (i).

\vs{.03in}
(iii)\quad We decompose the right hand side of the identity
\beq
  \label{comb}
{}\;\;\;\;\;
\op\big( g(a)-b \big)= \dim \vp \cdot 
\big\la g\cdot\op(a+\ten),\, \op(b+\ten) \big\ra
\quad (a,b\in A,\; g\in G)
\eeq
as a form of summation according to (ii).

\vs{.08in}
We remark that the identity (\ref{comb}) is easily deduced from the convolution identity
$\dis \op * \omega_{\calp'}=\delta_{\calp,\calp'} \frac{|\bbg|}{|\calp|}\op$ of zonal spherical functions (\cite{Mac}, section VII).

In the case of Dunkl\cite{Dun}, the decomposition in (i) is $G$-irreducible, and turns out to be multiplicity free and parametrized by a `single' sequence of integers.
For that reason, an organized addition theorem is achieved at the last step (iii).
In our case (on a local field in Section 1.2), however, the irreducible decomposition in (i) seems to be difficult.
And actually that is not necessarily multiplicity free.
So we use decompositions in the last section, that is not necessarily irreducible.
And we will assume some conditions for easy calculations.

\subsection{Construction of an addition theorem on some assumptions}

Along the procedure in the last subsection, here we propose a kind of addition theorem.
We take rather strong assumptions for easy calculations, nevertheless the result might be new and impressive.

We restrict the translating element $a\in R^N$ in the procedure (ii) to have the form of
\beq
1^t:=(\,\overbrace{1,\cdots, 1}^{t\text{-tuple}}
, \overbrace{0,\cdots, 0}^{(N-t)\text{-tuple}}
\,)\in R^N
\eeq
for a natural number $t\leq N$.
Then the condition ({\it ii}) in Proposition \ref{wau} forces that
\beq
   \label{acon}
\alpha_r^{(i)}=0\quad \text{for any $1\leq r \leq \ell-1,\; 2\leq i \leq I_{r-1}$,}
\eeq
furthermore, any $w=(w_{r,s})_{r\leq s} \in \mathcal{W}(\alpha, u)$ must be $w_{r,s}=0$ except for
\beq
t+|\alpha^{(1)}|-N\leq z:= w_{0,0} \leq t+|\alpha^{(1)}|-|n|
\eeq
from conditions ({\it i}), ({\it ii}), ({\it iii}) of Proposition \ref{wau}.
Then the proposition is reduced as follows:

\begin{col}
For $0\leq t \leq \ell-1$ and $\aia$ in the condition $(\ref{acon})$, we have
\beq
  \label{vnawnt}
\ba{l}\dis
V_{n,\al}\big(
\omega_n(1^t+\ten)
\big)
=\frac1{\binom{N}n Q}\1(
\prod\rl\;
\prod\iwa
\grzi{}^{\ai_r}
\2)
\sum_{z=t+|\ao|-N}^{t+|\ao|-|n|}
\1(
\gzzo
\2)^z
\\ \dis
\times\1[
\varphi_0\ueo{}^{\otimes z}
\otimes \bt_{r=0}^{\ell-1} \;\bt\iwa
\pri{}^{\otimes \al_r\uei} 
\2]_{\fraks_{t}}
\otimes%
\1[
\varphi_0\ueo{}^{\otimes\1( \ao_0-z \2)}
\otimes
\bt_{r=1}^{\ell-1} \pro{}^{\otimes\ao_r}
\2]_{\fraks_{N-t}}.
\ea
\eeq
\end{col}

We again restrict ourselves to deal with
\beq
c=(c_k)\kn\in(\ox)^N\subset G
\eeq
as the acting element $g\in G$ in the procedure (iii).
When $c$ acts on the function $V_{n,\al}\big(
\omega_n(1^t+\ten)
\big)$ in the last corollary,
if the $k$-th tensor element is $\pri$,
it is multiplied by $\xi\uei(c_k)$ as in (\ref{ri}).
Each term $\tf$ in (\ref{vnawnt}) is thus multiplied by
\beq
\prod\rl\;\prod\iwa\;\prod_{k\in\ari}\xi\uei(c_k),\quad\text{where}\;\;
\ari:=\1\{
k\in\{1,\cdots, t\}
\;\Big|\;
f_k=\pri
\2\}.
\eeq
Now let us calculate an inner product in the procedure (iii):

\begin{pp}
For $\aia$ satisfying $(\ref{acon}),\; 0\leq t\leq \ell-1,\;u=(u_s)_{s=0}^{\ell-1}\in \xln$ satisfying $t\leq u_0$ and $c\in (\ox)^N$, we have
\beq
  \label{the}
\ba{r}\dis
\bigg\la
c\cdot V_{n,\al}\big(
\omega_n(1^t+\ten)
\big)
,\;
V_{n,\al}\big(
\omega_n(\pi^u+\ten)
\big)
\bigg\ra
=
\frac1{\binom{N}n^2 Q^2}
\prod\rl\;
\prod\iwa
\1|\grzi\2|^{2\ai_r}
\\[.3in] \dis
\times \1(
\sum_{(\ari)_{r,i}}\; \prod\rl\; \prod\iwa\;\prod_{k\in\ari}
\xi\uei(c_k)
\2)
\\[.3in] \dis
\times
\prod_{s=1}^{\ell-1}
\sum_{k\geq0}
\binom{u_s}k
\binom{N-u|^s-\ao|_{s+1}}{\ao_s-k}
\gsso{}^k
\\[.3in] \dis
\times
\sum_{z\geq 0}
\sum_{k\geq 0}
\binom{t-|n|+|\ao|}{z}
\binom{u_0-t}{k}
\binom{N-u_0-\ao|_1}{\ao_0-z-k}
\gzzo{}^{2z+k},
\ea
\eeq
where the sum $\dis \sum_{(\ari)_{r,i}}$ in the second line runs over all disjoint subsets 
$\dis \bigsqcup\rl\;\bigsqcup\iwa \ari\subset \{1,\cdots, t \}$ such that $|\ari|=\ai_r$.
\end{pp}

\begin{proof}
The first line of the statement (\ref{the}) is obvious from (\ref{vnawnu}) and (\ref{vnawnt}).
We choose one term in the symmetrization $[\cdot]_{\fraks_t} \otimes [\cdot]_{\fraks_{N-t}}$ in (\ref{vnawnt}), that is, determine an order of the tensor as follows:
Firstly for each $i\geq 2$, we determine places where $\varphi_r^{(i)}$'s stay.
It yields the second line of (\ref{the}).
Secondly, in the decreasing order $s=\ell-1,\ell-2, \cdots, 1$, we choose places for $\varphi_s^{(1)}$'s in each terms of $[\cdot]_{\fraks_{N-t}}$ in (\ref{vnawnt}).
Among $\alpha_s^{(1)}$ number of them, let $k$ number be choosen from $\{ u|^{s-1}+1,\cdots, u|^{s} \}$.
Then the remaining $\alpha_s^{(1)}-k$ number must be taken from $\{u|^{s}+1, \cdots, N \} - \{ \text{indexes already taken for $\varphi_r^{(1)}$'s,\, $r>s$} \}$.
Since $\1(\gamma_{s,s}^{(1)}\2)^k$ is multiplied by (\ref{vnawnu}),
the third line of (\ref{the}) is gotten.
Finally, we choose places for $\varphi_0^{(1)}$'s.
We let $k$ number of them choosen from $\{ t+1,\cdots,N \}\cap \{1,\cdots,u_0\}$.
Then we conclude the last line of (\ref{the}).
Details are left to the reader.
\end{proof}

By substituting the values of $|\gamma_{r,0}^{(i)}|$ in Proposition \ref{maru} in the first line of (\ref{the}), we have
\beq
\prod\rl\;\prod\iwa\1| \grzi \2|^{2\ai_r}
=\frac1{q^{\sum\rl(r-1)\1(n_r-\ao_r\2)}
(q-1)^{2\1(|n|-|\ao|\2)}}.
\eeq

By the definition (\ref{kraw}) of the Krawtchouk polynomials, we notice that the third and fourth lines of (\ref{the}) are in the form of that polynomials.
Also recall that $\gsso=-\frac1{q-1}$ for $0\leq s\leq \ell-1$ (Proposition \ref{maru}).
Together with the definition (\ref{lkraw}) of the $\ell$-variate Krawtchouk polynomials, we have
\beq
\ba{l}
\text{(the 3rd and 4th lines of (\ref{the}))}
\\[0.2in] \dis
=\prod_{s=1}^{\ell-1}
\binom{N-u|^{s-1}-\ao|_{s+1}}{\ao_s}
\;K_{\ao_s}\1( u_s; \!\!\ba{c}\frac{q-1}q\ea\!\!;\, N-u|^{s-1}-\ao|_{s+1} \2)
\\[0.2in] \dis
\times \sum_{z\geq0}
\1(\!\!\!
\ba{c}-\frac1{q-1}\ea
\!\!\!\2)^{2z}
\binom{t-|n|+|\ao|}{z}
%
\binom{N-t-\ao|_1}{\ao_0-z}
K_{\ao_0-z}\1( u_0-t; \!\!\ba{c}\frac{q-1}q\ea\!\!; N-t-\ao|_1 \2)
\\[0.4in] \dis
=\sum_{z\geq0}\frac1{(q-1)^{2z}}
\binom{t-|n|+|\ao|}{z}
\binom{N-t}{\ao_0-z,\, \ao_1,\cdots,\ao_{\ell-1}}
\\[0.2in] \dis
\quad\times\; K\uel_{\1( \ao_0-z,\, \ao_1,\cdots,\ao_{\ell-1} \2)}
\1( u_0-t,\, u_1,\cdots,u_{\ell-1}; \!\!\ba{c}\frac{q-1}q\ea\!\!; N-t \2).
\ea
\eeq

\vs{0.1in}
\nt
Now we take the sum of (\ref{the}) over $\aia$ satisfying (\ref{acon}).
By splitting the sum as
\beq
   \label{swa}
\sum_{\aia}=\sum_{\{\ao_r\}\rl}\;\sum_{\{ \ai_r \}_{r,i}},
\eeq
then the second sum of (\ref{swa}) and the second line of (\ref{the}) yield $\ell$-variate Krawtchouk polynomials again;
\beq
\ba{l} \dis
\hs{-00.5in}
\sum_{\{ \ai_r \}_{r,i}}\1(
\sum_{(\ari)_{r,i}}\; \prod\rl\; \prod\iwa\;\prod_{k\in\ari}
\xi\uei(c_k)
\2)
\\[.4in] \dis
=\binom{t}{n-\ao}
q^{\sum_{r=1}^{\ell-1}(r-1)\1(n_r-\ao_r\2)}
(q-1)^{2\1(n|_1-\ao|_1\2)}
(q-2)^{n_0-\ao_0}
\\[.2in] \dis
\quad \times\;
K\uel_{n-\ao}
\1(y;
\ba{c}
\frac{q-2}{q-1}, \frac{q-1}q,\cdots,\frac{q-1}q
\ea
;\;t
\2),
\ea
\eeq
where $y=(y_r)\rl$ is the set of numbers that reflects the group action, defined by
\beq
y_r= \sharp\1\{
k\in\{1,\cdots, t\} \;\Big|\;
c_k\in \ox_{r-1} - \ox_r
\2\}
\quad(0\leq r\leq \ell-1).
\eeq

On the other hand, by counting numbers of components that have each value of the valuation $v$,
we know
\beq
c\cdot 1^t - \pi^u \in \calo\big(
u_0-t+y_0,\, u_1+y_1, \cdots, u_{\ell-1}+y_{\ell-1}
\big).
\eeq
So in our situation, the left hand side of (\ref{comb}) is
\beq
K_n\uel\1(
u_0-t+y_0,\, u_1+y_1, \cdots, u_{\ell-1}+y_{\ell-1};
\ba{r} \frac{q-1}q \ea
;\;N
\2).
\eeq

All the observations above induce the main theorem of the paper.
In the theorem,
we use the following notation:

\vs{0.1in}
\nt\underline{Notation.}\; For a sequence $x=(x_r)\rl$ and a number $a$,
we denote by $x+ae_0$ the sequence added by $a$ in the only 0th element;
\beq
x+ae_0=\1(x_0+a,\, x_1,\cdots,x_{\ell-1}\2).
\eeq
And in the theorem, we replace $\ao=(\ao_r)\rl$ by $\al=(\al_r)\rl$ and other letters are used in the same way as in the discussion so far.

\np
\begin{thrm}
For $n\in\xln,\; 0\leq t\leq N,\; u\in\xln$ such that $t\leq u_0$, and $y\in X(\ell, t)$, we have
$$
\ba{l}\dis
K_n\uel\1(
u+y-te_0\;;
\ba{r} \frac{q-1}q \ea
;\;N
\2)
\\[0.2in] \dis
=\frac1{\binom{N}n}
\sum_{\al\in\xln}\;\sum_{z\geq0}
\1(\!\! \ba{r} \frac{q(q-2)}{(q-1)^2} \ea \!\!\2)^{n_0-\al_0}
\frac1{(q-1)^{2z}}
\binom{t}{n-\al, z} \binom{N-t}{\al-ze_0}
\\[0.3in] \dis
\quad\times\; K\uel_{\al-ze_{{}_{0}}}
\1(
u-te_0;
\ba{r} \frac{q-1}q \ea
;\;N-t
\2)
\; K\uel_{n-\al}
\1(
y;
\ba{c}
\frac{q-2}{q-1}, \frac{q-1}q,\cdots,\frac{q-1}q
\ea
;\;t
\2).
\ea
$$
\end{thrm}


\end{document}